\documentclass[letterpaper, conference]{IEEEtran}
\usepackage{fullpage,graphicx,url,setspace}

\usepackage{xparse}
\usepackage{cite}
\usepackage{bm}
\usepackage{color}
\usepackage[small,bf]{caption}
\usepackage{amsmath,verbatim,amssymb,amsfonts,amscd, graphicx, algorithmic, algorithm}
\usepackage{amsmath,amsthm}
\usepackage{mathtools}
\mathtoolsset{mathic}
\usepackage{microtype}
\usepackage{wrapfig}

\usepackage{ifxetex,ifluatex}
\ifxetex
\usepackage{xltxtra}
\else
\ifluatex
\usepackage{luatextra}
\else
\usepackage{amssymb}
\usepackage{algorithmic, algorithm}
\usepackage[utf8]{inputenc}

\csname else\expandafter\endcsname
\romannumeral -`0
\fi
\fi
\iftrue\relax%
\defaultfontfeatures{Ligatures=TeX}
\usepackage[math-style=TeX, vargreek-shape=TeX]{unicode-math}
\fi

\NewDocumentCommand{\entropy}{om}{\mathbb{H}\left[#2
    \IfValueT{#1}{\,\middle|\,#1}\right]}
\NewDocumentCommand{\bentropy}{lm}
  {\widetilde{\mathbb{H}}#1\left[#2\right]}

\NewDocumentCommand{\mutualInfo}{omm}{\mathbb{I}\left[#2;#3
    \IfValueT{#1}{\,\middle|\,#1}\right]}

\usepackage{tikz}
\usetikzlibrary{shapes.geometric, positioning, shadows}
\usepackage{graphicx,color}

\newtheorem{proposition}{Proposition}
\newtheorem{lemma}{Lemma}

\usepackage{enumitem}
\newlist{enumerate*}{enumerate*}{1}
\setlist[enumerate*]{label=(\arabic*)}

\newcommand{\ben}{\begin{eqnarray}}

\newcommand{\een}{\end{eqnarray}}

\title{Sequential detection of low-rank changes \\using extreme eigenvalues}

\author{\IEEEauthorblockN{Liyan Xie and Yao Xie}
\IEEEauthorblockA{H. Milton Stewart School of Industrial and Systems Engineering, Georgia Institute of Technology\\
\{lxie49, yao.xie\}@gatech.edu}
}

\begin{document}
\maketitle

\begin{abstract}
We study the problem of detecting an abrupt change to the signal covariance matrix. In particular, the covariance changes from a ``white'' identity matrix to an unknown spiked or low-rank matrix. Two sequential change-point detection procedures are presented, based on the largest and the smallest eigenvalues of the sample covariance matrix. To control false-alarm-rate, we present an accurate theoretical approximation to the average-run-length (ARL) and expected detection delay (EDD) of the detection, leveraging the extreme eigenvalue distributions from random matrix theory and by capturing a non-negligible temporal correlation in the sequence of scan statistics due to the sliding window approach. Real data examples demonstrate the good performance of our method for detecting behavior change of a swarm.
\end{abstract}

\vspace{-0.05in}
\section{Introduction}
\vspace{-0.05in}

Detecting change-points from high-dimensional streaming data is a fundamental problem in various applications such as video surveillance, sensor networks, and biological swarm behavior study. In various scenarios, the change happens to the covariance. For instance, before the change, the signal covariance matrix is an identity matrix, and after the change, the signal covariance matrix becomes a spike low-rank matrix \cite{berthet2013optimal}, or becomes a low-rank matrix due to the emergence of a subspace structure \cite{xie2016sequential}.  

To detect such changes, it is natural to consider extreme (largest or smallest) eigenvalues of the sample covariance matrices over a time sliding window. Our prior work  \cite{xie2016sequential} presents an initial study of such method using the largest eigenvalue. In this paper, we present a more thorough theoretical analysis of these procedures with largest and smallest eigenvalues. In controlling the false alarm rate, i.e., deriving the average run length (ARL), we take a different approach than \cite{xie2016sequential}  by leveraging random matrix theory. In particular, we use the Tracy-Widom law for extreme eigenvalues to derive an ARL approximation. We further refined the precision our ARL approximation by capturing the temporal dependence of the detection statistic which is an inherent problem that needs to be addressed for change-point detection procedure due to its sliding window approach. Our technique is based on change-of-measure for  Gaussian random field developed in \cite{siegmund2010tail}. Real data examples demonstrate the good performance of our algorithm.

The rest of this paper is organized as follows. Section II sets up the formalism for low-rank change detection. Section III presents the low-rank detection performance after ignoring temporal correlation of scan-statistics, and show an approximation to the temporal correlation of scan-statistic and give a more accurate theoretical analysis. Section IV presents the numerical example to demonstrate the performance of our method. Proofs can be found in the arXiv version this paper.

 \section{Setup and detection procedures}

Assuming a sequence of $p$-dimensional vectors $x_1, x_2, \ldots, x_t$, $t = 1, 2, \ldots$. There may be a change-point at time $\tau$ such that the distribution of the data stream changes. Assume the change happens at the covariance structure of the matrix. We will consider two related problems arising from specific applications: (i) the post-change covariance matrix model is a spiked covariance matrix; (ii) the covariance matrix before the change is full rank, and after the change becomes low-rank. Our goal is to detect such a change as quickly as possible.

\subsection{Largest eigenvalue procedure}

Formally, problem (i) can be stated as the following hypothesis test:
\begin{equation}
\left\{
\begin{array}{ll}
\textsf{H}_0:  & x_1, x_2, \ldots, x_t \stackrel{iid}{\sim} \mathcal{N}(0, I_p) \\
\textsf{H}_1: & 
x_1, x_2, \ldots, x_\tau \stackrel{iid}{\sim} \mathcal{N}(0, I_p),  \\
&~~~x_{\tau+1}, \ldots, x_t \stackrel{iid}{\sim} \mathcal{N}(0,  I_p + \theta u u^\intercal).
\end{array}\right. 
\label{eq:1}
\end{equation}
where $u \in \mathbb{R}^{p\times 1}$ represents a basis for a subspace and $u^\intercal u=1$. 


One may construct a maximum likelihood ratio statistic. However, since the signal covariance matrix or $u$ are unknown, we may have to form the generalized likelihood ratio statistic, which replaces the covariance matrix with the sample covariance. This may cause an issue since the statistic involves inversion of the sample covariance matrix, whose numerical property (such as condition number) is usually poor when $p$ is large. 

Alternatively, we consider the largest eigenvalue of the sample covariance matrix which is a natural detection statistic here. We adopt a scanning window approach to sequential methods. Consider samples in a time window of $[t-w, t]$ at each time $t$, where $w$ is the window size. Form the sample covariance matrix
\begin{equation}
\widehat{\Sigma}_{t-w, t} = \frac{1}{w} \sum_{i=t-w+1}^t x_i x_i^\intercal. \label{sampleCov}
\end{equation}
Using the largest eigenvalue of the sample covariance matrix, we form the {\it maximum eigenvalue procedure}, which is a stopping time given by:
\begin{equation}
\label{procedure1}
T_1 = \inf\{t: \lambda_{\max}(\widehat{\Sigma}_{t-w,t}) \geq b \}, 
\end{equation}
where $b > 0$ is the threshold, and $\lambda_{\max}(\Sigma)$ denotes the largest eigenvalue of a matrix $\Sigma$. An alarm is fired whenever the detection statistic exceeds the threshold $b$.

\vspace{-0.1in}
\subsection{Smallest eigenvalue procedure}

Problem (ii) can be formulated as 
\begin{equation}
\left\{
\begin{array}{ll}
\textsf{H}_0:  & x_1, x_2, \ldots, x_t \stackrel{iid}{\sim} \mathcal{N}(0, I_p) \\
\textsf{H}_1: & 
x_1, x_2, \ldots, x_\tau \stackrel{iid}{\sim} \mathcal{N}(0, I_p),  \\
&~~~x_{\tau+1}, \ldots, x_t \stackrel{iid}{\sim} \mathcal{N}(0,  \theta u u^\intercal).
\end{array}\right.
\label{H2}
\end{equation}
where $u \in \mathbb{R}^{p\times 1}$ represents the basis for a one-dimensional subspace, $u^\intercal u=1$, and $\theta > 0$ represents the signal power. 

For this problem, we consider the smallest eigenvalue of the sample covariance matrix as the testing statistic, since the change is from full-rank to low-rank. We also adopt a scanning window approach assuming a window size $w$. Using the smallest eigenvalue of the sample covariance matrix, we form the {\it minimum eigenvalue procedure}, which is a stopping time given by:
\begin{equation}
\label{procedure2}
T_2 = \inf\{t: [\lambda_{\min}(\widehat{\Sigma}_{t-w,t})]^{-1} \geq b \}, 
\end{equation}
where $b$ is a pre-specified threshold.  
 
\subsection{Choice of window-size $w$} There are two considerations when choosing $w$. First, $w$ has to be larger than the anticipated detection delay to not losing performance.  
Intuitively, we know that the signal strength $\theta$ need to be large enough in order to tell apart $\textsf{H}_0$ and $\textsf{H}_1$.  The result in \cite{baik2006eigenvalues} shows that there is a critical value of signal strength $\theta$, i.e., when $\theta$ is above $\sqrt{p/n}\sigma$ the largest eigenvalue can be separated from the background noise. In our setting, this translate to that we need the window length needs to satisfy $\theta \geq \sigma \sqrt{p/w}$, which poses a requirement on the minimum window length $w \geq p/(\theta^2/\sigma^2)$, where $\theta^2/\sigma^2$ can be interpreted as the signal-to-noise ratio.
Second, $w$ should be chosen greater than the anticipated longest detection delay.

\section{Theoretical analysis}

Two commonly used key performance metrics include the (i) the average run length (ARL), denoted as $\mathbb{E}^\infty[T_i]$, $i = 1, 2$, which is the expected duration in between two false alarms when there is no change, and (ii) the expected detection delay (EDD), denoted as $\mathbb{E}^1[T_i]$, $i = 1, 2$, which is expected number of samples before the procedure stops when the change-point happens at the first time $t = 1$. In the following, we characterize these two metrics theoretically. Usually, the threshold $b$ is chosen such the ARL meets a certain large targeted value (e.g., 5,000 or 10,000). 

\subsection{Background on extreme eigenvalue distributions}

Since our detection procedures are based on extreme eigenvalues of the sample covariance matrix, it is essential to review their distributions from random matrix theory.  There are two kinds of results typically available for eigenvalue distributions: one for the so-called {\it bulk}, which refers to the properties of the full set of eigenvalues, and one for the {\it extremes}, which are the (first few) largest and smallest eigenvalues. 

For bulk spectrum, two well-known results are the the Wigner's semicircle law (see, e.g., \cite{edelman2013random}), which describes the limiting density of eigenvalues of square symmetric random matrices, and the Marchenko-Pastur law for covariance matrices.  

Based on the bulk spectrum distribution, various results have been obtained for extreme eigenvalue distributions. 
Assume there are $n$ samples are i.i.d. $p$-dimensional Gaussian random vectors with zero-mean and identity covariance matrix. Let the sample covariance matrix be $\widehat{\Sigma}_{n}=\frac1n\sum_{i=1}^n x_i x_i^\intercal$. \cite{geman1980limit} shows that if $p/n \rightarrow \gamma >0$, the largest eigenvalue of the sample covariance matrix converges to $(1+\sqrt{\gamma})^2$ almost surely. 
Although this is a very good approximation, it does not describe the variability of the largest eigenvalue. To characterize the distribution of the largest eigenvalues, \cite{johnstone2001distribution} uses the Tracy-Widom law \cite{TracyWidom96}. Define the center and scaling constants 
\begin{equation}
\begin{aligned}
 \mu_{np}&=(\sqrt{n-1}+\sqrt{p})^2, \\
 \sigma_{np}&=(\sqrt{n-1}+\sqrt{p})(\frac1{\sqrt{n-1}}+\frac1{\sqrt{p}})^{1/3}. 
 \end{aligned}
 \label{eq:parameter}
 \end{equation} 
 If $p/n \rightarrow \gamma <1$, then the centered and scaled largest eigenvalue converges in distribution to the so-called Tracy-Widom law of order one $W_1 \sim F_1$ \cite{johnstone2001distribution}:
\begin{equation}
 \frac{\lambda_{\max}(\widehat{\Sigma}_{n})-\mu_{np}/n}{\sigma_{np}/n} \rightarrow W_1 \sim F_1,
 \label{eq:tw} 
 \end{equation}
The Tracy-Widom law can be described in terms of partial differential equation and the Airy function, and its tail can be computed conveniently using an R-package \textsf{RMTstat}. 

Similar result exists for the smallest eigenvalue of the sample covariance matrix \cite{feldheim2010universality}:
\[ \frac{\lambda_{\min}(\widehat{\Sigma}_{n})-\mu^\prime_{np}/n}{\sigma^\prime_{np}/n} \rightarrow W_1 \sim F_1, \]
where 
\[ \mu^\prime_{np}=(\sqrt{p}-\sqrt{n})^2, \] 
\[ \sigma^\prime_{np}=(\sqrt{p}-\sqrt{n})(\frac1{\sqrt{p}}-\frac1{\sqrt{n}})^{1/3}. \]

\subsection{Approximation to Average-Run-Length (ARL)}

Note that the scan-statistics in the Procedures (\ref{procedure1}) over time has temporal correlation due to overlapping data. For instance, $\widehat{\Sigma}_{t-w-1, t-1}$ and $\widehat{\Sigma}_{t-w, t}$ both involve samples $\{x_{t-w+1}, \ldots, x_{t-1}\}$. 
To simplify the analysis, we first ignore the temporal correlation of scan-statistics and suppose different sample covariance matrixes $\widehat{\Sigma}_{t-w,t}$ are independent even if they may have some overlaps.

The Tracy-Widom law above gives an asymptotic distribution of the largest eigenvalue, which is very useful for us to analyze the distribution of our scan-statistics and choose the threshold of the detection procedure.

\begin{proposition}[Approximation to ARL based on Tracy-Widom, ignore temporal correlation]
For any $\alpha \in (0, 1)$, when choosing for  $T_1$
\begin{equation}
b = \frac{\sigma_{wp}}{w} \mathcal{T}_\alpha + \frac{\mu_{wp}}{w}, 
\label{eq:threshold1}
\end{equation}
 $\mathbb{E}^\infty[T_1]\approx 1/\alpha$. Above, $\mathcal T_\alpha$ is the $\alpha$-upper-percentage point for Tracy-Widom law of order one.
When choosing for $T_2$, 
\begin{equation}
b = \left[\frac{\sigma^\prime_{wp}}{w} \mathcal T_{1-\alpha}+\frac{\mu^\prime_{wp}}{w}\right]^{-1},
\label{eq:threshold2}
\end{equation}
$\mathbb{E}^\infty[T_2]\approx 1/\alpha$.  
\label{ARL1}
\end{proposition}

\noindent{\it Remark:} 
Note that the threshold choice here is different from that in \cite{berthet2013optimal}, since here we used the tail probability of the precise distribution of the extreme eigenvalues based on Tracy-Widom law, whereas the later uses an approximate threshold to characterization the rate of change for extreme eigenvalues, which may be less accurate.

\subsection{Correlation of temporal scan-statistics}

Now we aim to capture the temporal dependence in the scan statistics, which may be significant due to overlapping of adjacent time windows. To start, we consider $T_1$ which uses the largest eigenvalue. This is a very challenging task, as it is an open question what is the correlation between the largest eigenvalues of two sample covariance matrices share partially common data.  

For sample covariance matrix defined over a time window $[t-w, t]$ in \eqref{sampleCov}, define
\begin{equation}
Z_{t}=\lambda_{\max}(\widehat{\Sigma}_{t-w,t}).
\end{equation}
Define the correlation between two random variables as usual \[{\rm corr}(X, Y) = \frac{\mathbb{E}[X Y] - \mathbb E(X) \mathbb E(Y)}{\sqrt{{\rm Var}(X)}\sqrt{{\rm Var}(Y)}}.\]

\begin{figure}[h!]
\begin{center}
\includegraphics[width = 0.2\textwidth]{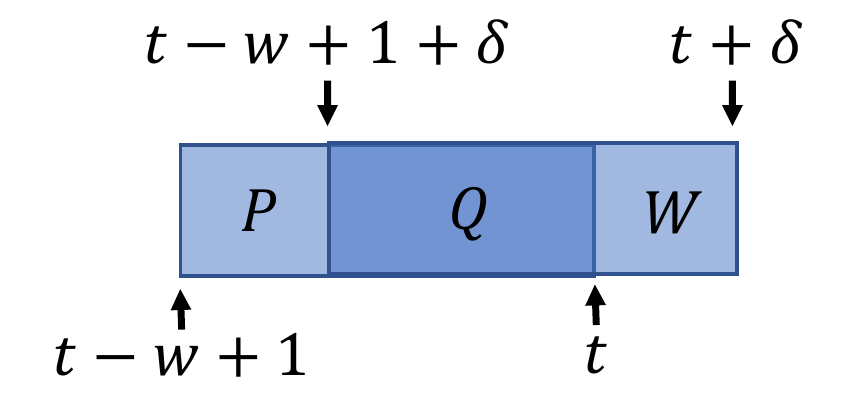} 
\end{center}
\caption{Illustration for computing covariance between the largest eigenvalues formed by sample covariance matrices formed by overlapping data.}
\vspace{-0.1in}
\end{figure}
\vspace{-0.1in}
We have the following lemma, which leads to some insights for our analysis. 
\begin{lemma}[Approximation to local correlation]\label{CovZ}
Under the null hypothesis, $x_t \stackrel{iid}{\sim} \mathcal{N}(0, I_p)$, For $\delta \in \mathbb{Z}^+$, let 
\[
P=\sum\limits_{i=t-w+1}^{t-w+\delta}x_ix_i^\mathrm{T},  
Q=\sum\limits_{i=t-w+\delta+1}^t x_ix_i^\mathrm{T},  
W=\sum\limits_{i=t+1}^{t+\delta} x_ix_i^\mathrm{T},
\]
then $P$, $Q$ and $W$ are mutually independent random matrices. We have that
\begin{equation*}
\begin{aligned}
\mathbb{E}[Z_{t}Z_{t+\delta}]
&\leq \frac1{w^2}\{\mathbb{E}[\lambda_{\max}(Q)^2]  \\
& +\mathbb{E}[\lambda_{\max}(Q)] (\mathbb{E}[\lambda_{\max}(P)] + E[\lambda_{\max}(W)])\\
&+\mathbb{E}[\lambda_{\max}(P)]\mathbb{E}[\lambda_{\max}(W)]\}
\end{aligned}
\end{equation*}
where the mean and second-order moments can be computed using Tracy-Widom law shown in \eqref{eq:tw}. Furthermore, an approximation to the local correlation for $w \gg p$ and small $\delta \ll w$, is given by
\begin{equation}
{\rm corr}(Z_t,Z_{t+\delta}) \lesssim 1 - \left(1+\frac{2p^{\frac13}+3p^{\frac16}\frac{c_1}{\sqrt{w}}+\frac{c_1^2}{w}}{c_2^2}\right)\delta + o(\delta)
\label{corr_approx}
\end{equation}
where $c_1 = \mathbb{E}(W_1)=-1.21$ and $c_2=\sqrt{{\rm Var}(W_1)}= 1.27$. 
\end{lemma}


\begin{table}[H]
\begin{center}
\caption{Comparison of upper bound for $\mathbb{E}[Z_{t}Z_{t+\delta}]$ in Lemma \ref{CovZ} with simulation, for $w=200$ and $p=10$.}
\begin{tabular}{|c|c|c|c|c|c|}
\hline
&$\delta=2$&$\delta=6$&$\delta=10$&$\delta=15$&$\delta=20$\\
\hline
simulation &1.980    &    1.980 &  1.979  & 1.979   &1.978     \\
\hline
upper bound & 2.144 &   2.241  &   2.294  & 2.342  &  2.379 \\
\hline
\end{tabular}
\label{upperbound}
\end{center}
\vspace{-0.1in}
\end{table}

Lemma \ref{CovZ} offers reasonably good approximation the the $\mathbb{E}[Z_{t}Z_{t+\delta}]$, as shown in Table \ref{upperbound}. 
We also run simulations to verify correlation between the largest eigenvalues between two overlapping data blocks ${\rm corr}(Z_t, Z_{t+\delta})$, shown in Fig. \ref{fig:cov_ratio}.
Note that the correlation increases as the ``overlapping ratio'' $1-\delta/w$ increases. Indeed, the correlation function can be linearly approximated in a local region when $\delta/w \ll 1$.

%
\begin{figure}[h]
\begin{center}
\includegraphics[width = 0.25\textwidth]{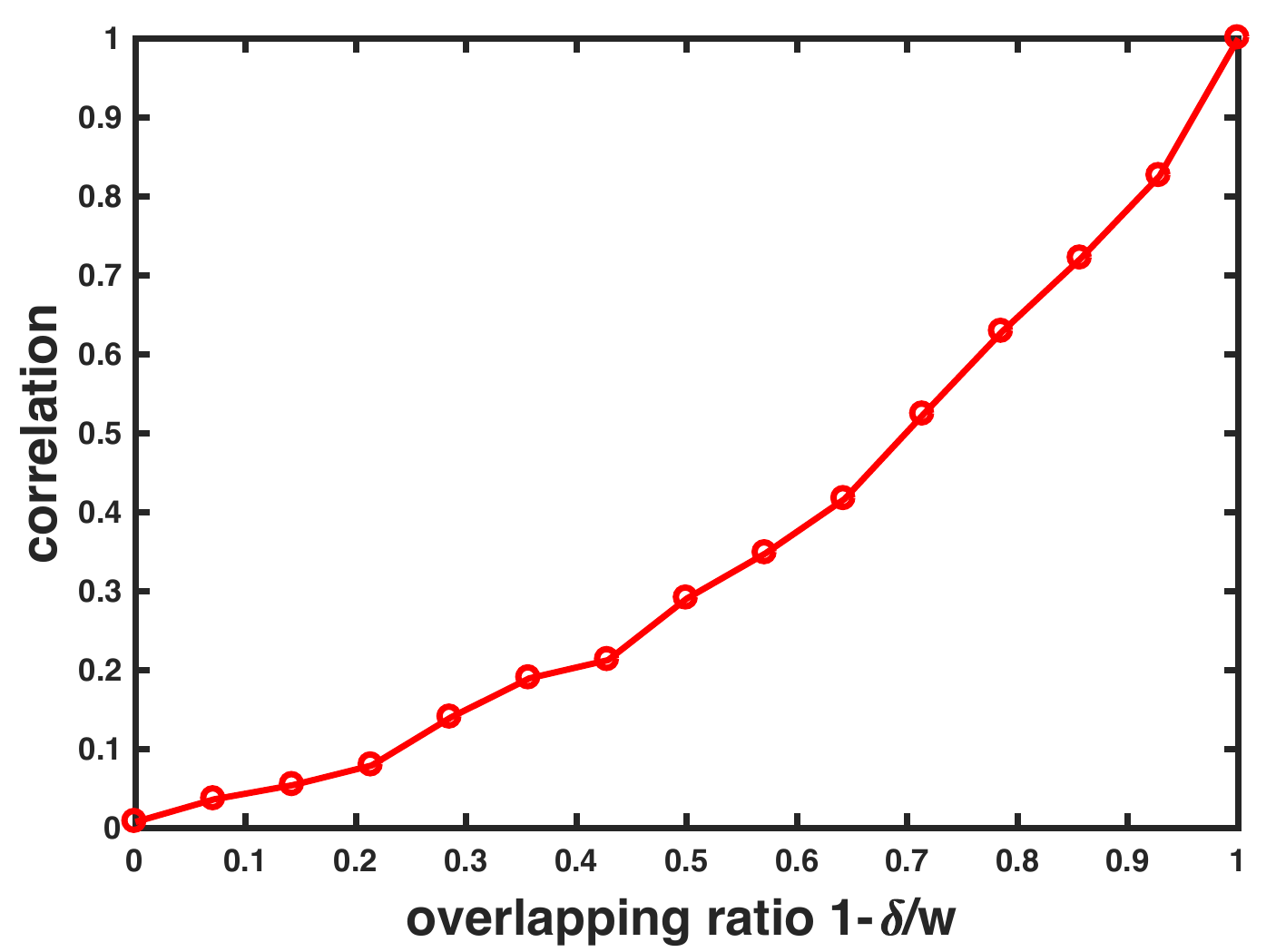} 
\end{center}
\caption{Simulated ${\rm corr}(Z_t, Z_{t+\delta})$ as a function of the ``overlapping-ratio'' $1-\delta/w$, for $w = 200$ and $p = 10$. 
}
\label{fig:cov_ratio}
\end{figure}

\begin{proposition}[Approximation to ARL, considering temporal correlation]
When $b\rightarrow\infty$,  if ${\rm corr}(Z_{t}, Z_{t+\delta}) =   1- \beta \delta + o(\delta)$, then
the average-run-length (ARL) of $T_1$ in \eqref{procedure1} can be approximated as 
\begin{equation}
\mathbb{E}^\infty[T_1] \approx [\beta b \phi(b) v(b\sqrt{2\beta})]^{-1},
\label{eq:ARL}
\end{equation}
Above, $v(\cdot)$ is a special function closely related to the Laplace transform of the overshoot over the boundary of a random walk (\cite{siegmund2007statistics}):
\[
v(x) \approx \frac{\frac2x\left[\Phi(\frac x2)-\frac12\right]}{\frac x2\Phi(\frac x2)+\phi(\frac x2)}, 
\]
and $\phi(x)$ and $\Phi(x)$ are the probability density function and the cumulative density function of the standard normal distribution. 
\label{ARL2}
\end{proposition}
\noindent The proof uses the change-of-measure techniques in \cite{siegmund2010tail}.

Now we verify the accuracy of the threshold obtained without and with considering the temporal correlation (Proposition \ref{ARL1} and Proposition \ref{ARL2}, respectively). 
We find that  indeed the threshold with temporal correlation correlation \eqref{eq:ARL} with $\beta$ in \eqref{corr_approx} is more accurate than the threshold obtained from Tracy-Widom law ignoring temporal correlation (which seems to perform reasonable well).
%
\begin{table}[H]
\caption{Comparison of the threshold $b$ obtained from simulation and the approximation. Window length $w=200$, dimension $p = 10$. Simulation is the reference threshold. 
}
\begin{tabular}{|c|c|c|c|c|c|c|}
\hline
Target ARL & 5k  &  10k  &  20k  & 30k &  40k & 50k\\
\hline
Simulation   &    1.633  &  1.661 &   1.688   & 1.702   & 1.713  &  1.722   \\
\hline
TW approx \eqref{eq:threshold1} & 1.738 & 1.763 & 1.787 & 1.800 &  1.809 & 1.816\\
\hline
\eqref{eq:ARL}, $\beta$ in \eqref{corr_approx} & 1.699  &  1.713  &  1.727   & 1.735 &   1.740  &  1.744 \\
\hline
\end{tabular}
\label{eq:threshold}
\end{table}


\subsection{Approximation to EDD}


A lower bound to EDD is computed using the same strategy as in \cite{xie2016sequential} but here the approximation is more accurate.

\begin{proposition}[Lower bound to EDD]
When $b\rightarrow \infty$
\begin{equation}
\mathbb{E}^1[T] \gtrsim \frac{b^\prime + e^{-b^\prime}-1}{\frac{\theta}{2}-\frac12\log(1+\theta)}.
\label{eq:EDD}
\end{equation}
where $b^\prime = \frac12[1-\frac1{1+\theta}][b-\frac{\log(1+\theta)}{1-1/(1+\theta)}]w$.
\label{EDD}
\end{proposition}
The proof is based on an result for the CUSUM procedure in \cite{siegmund2013sequential}.  
Consistent with intuition, Proposition \ref{EDD}, the right-hand-side of (\ref{eq:EDD}) is a decreasing function of $\theta$, representing the signal-to-noise ratio. 
Moreover, we compare the lower bound in Proposition \ref{EDD} with simulated average delay, as shown in Fig. \ref{fig:EDD}. In the regime of small detection delay (which is the main regime of interest), the lower bound serves as a reasonably good approximation. There is still room for improvement, which we are working on leveraging Grothendieck's Inequality \cite{guedon2016community}. 
\begin{figure}[h]
\begin{center}
\includegraphics[width = 0.25\textwidth]{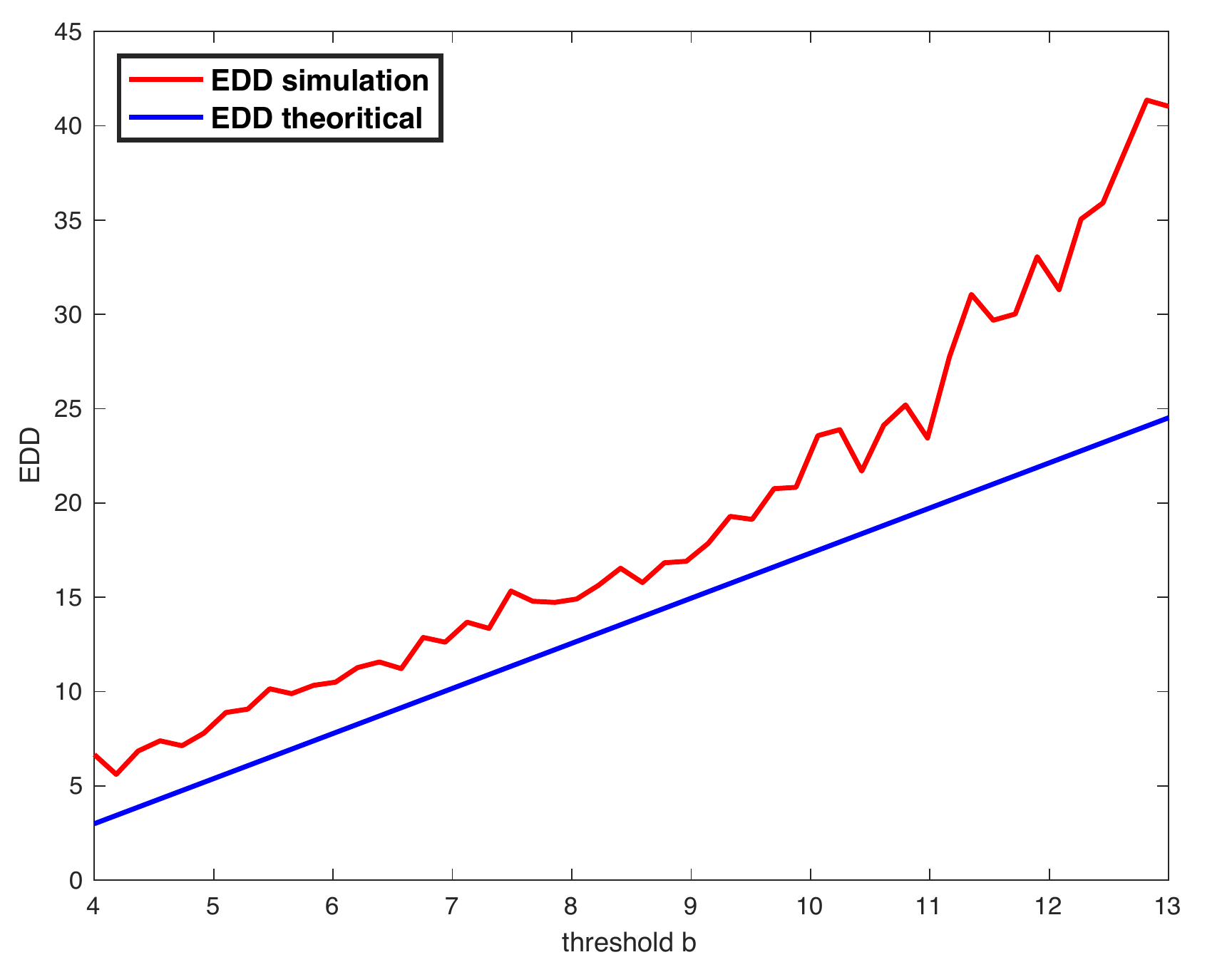} 
\end{center}
\caption{The comparison of the lower bound on EDD and the simulated EDD corresponding to different thresholds $b$, here the window length $w=20$, signal strength $\theta=10$, dimension $p=10$.}
\label{fig:EDD}
\end{figure}

\section{Real-data examples}


In this section we present two examples for detecting swarm behavior changes \cite{berger2016classifying}. We are interested in detecting the transition of a swarm (represented by dots in Fig. \ref{fig:swarm} by dots) transitioning from random behavior to some organized behavior: such as ``flock'' or ``torus''. The swarm dataset contains the coordinates of each individual at each unit time between a specific time interval. 
%
The swarm dataset contains a sequence of the coordinates and velocity of each individual at each time. We verify that in these cases the change can be well represented by the setting in \eqref{H2}, and applying $T_2$ in \eqref{procedure2} can quickly detect the transition from random behavior to ``flock'' or ``torus''.

\begin{figure}[h!]
\begin{center}
\begin{tabular}{c}
\includegraphics[width = 0.13\textwidth]{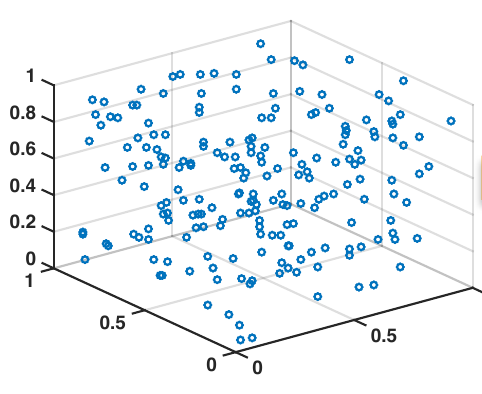} 
\includegraphics[width = 0.13\textwidth]{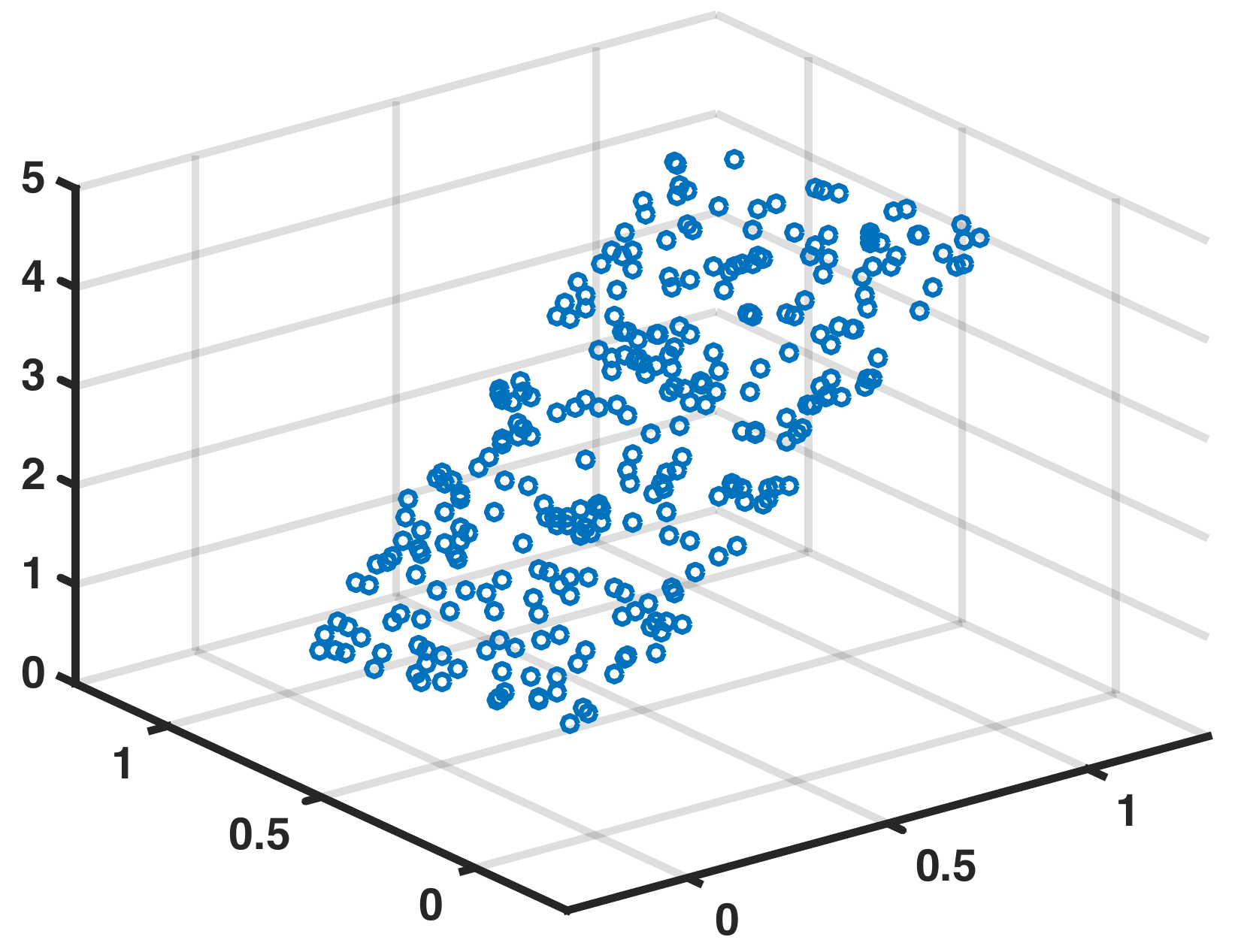}
\includegraphics[width = 0.15\textwidth]{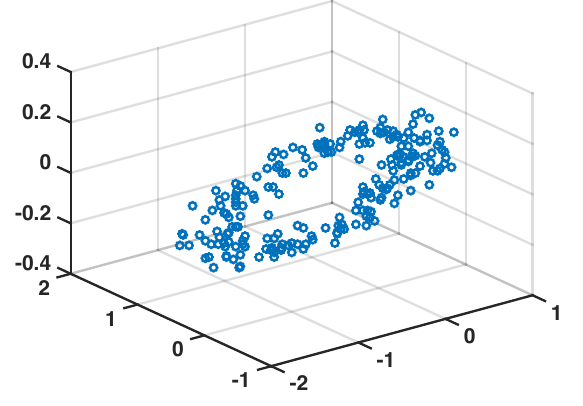}
\end{tabular}
\end{center}
\vspace{-0.1in}
\caption{Pattern of swarm (from left to right): random, flock (move along the same direction); torus.}
\label{fig:swarm}
\vspace{-0.1in}
\end{figure}

\begin{figure}[h!]
\begin{center}
\begin{tabular}{cc}
\includegraphics[width = 0.15\textwidth]{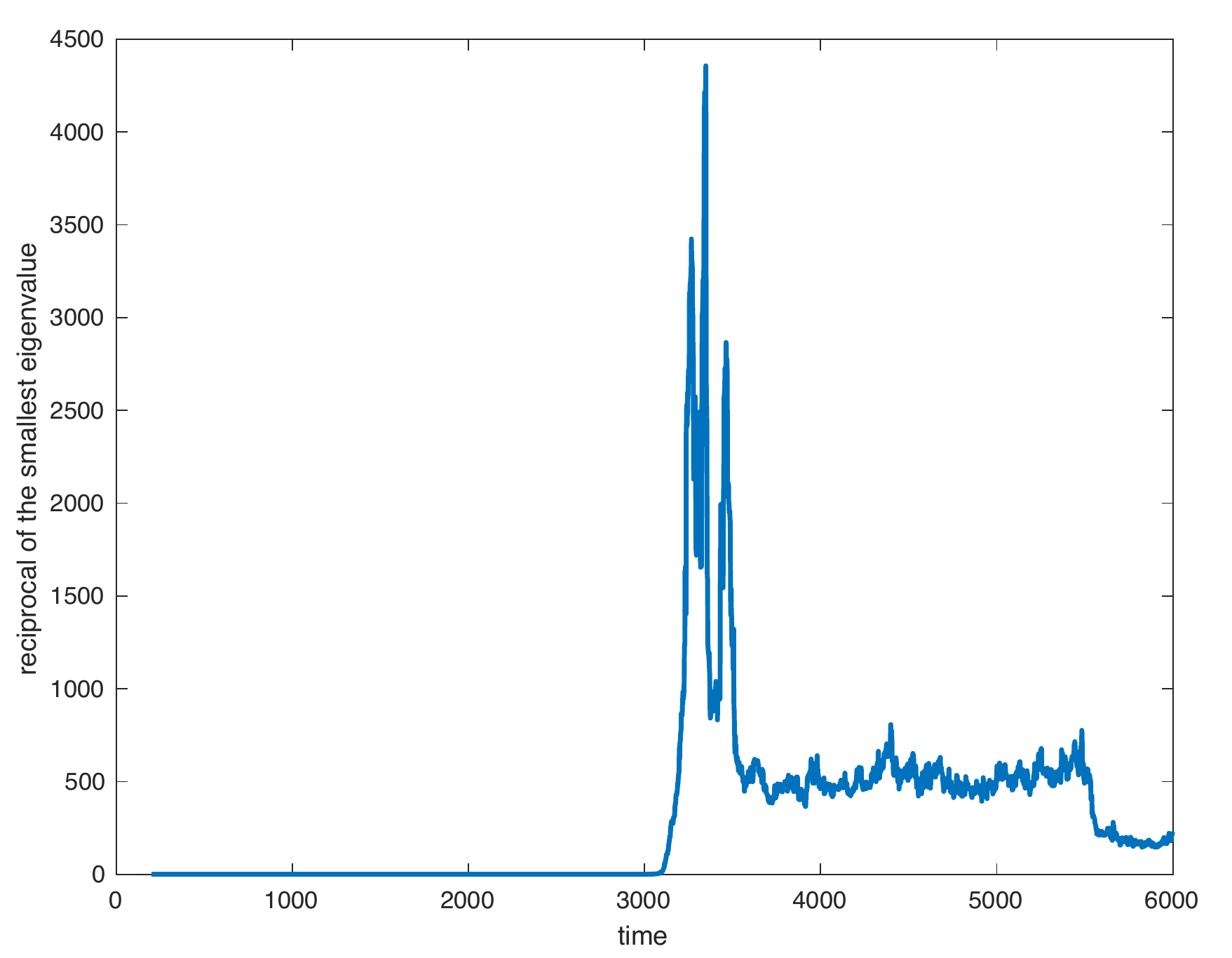} & 
\includegraphics[width = 0.15\textwidth]{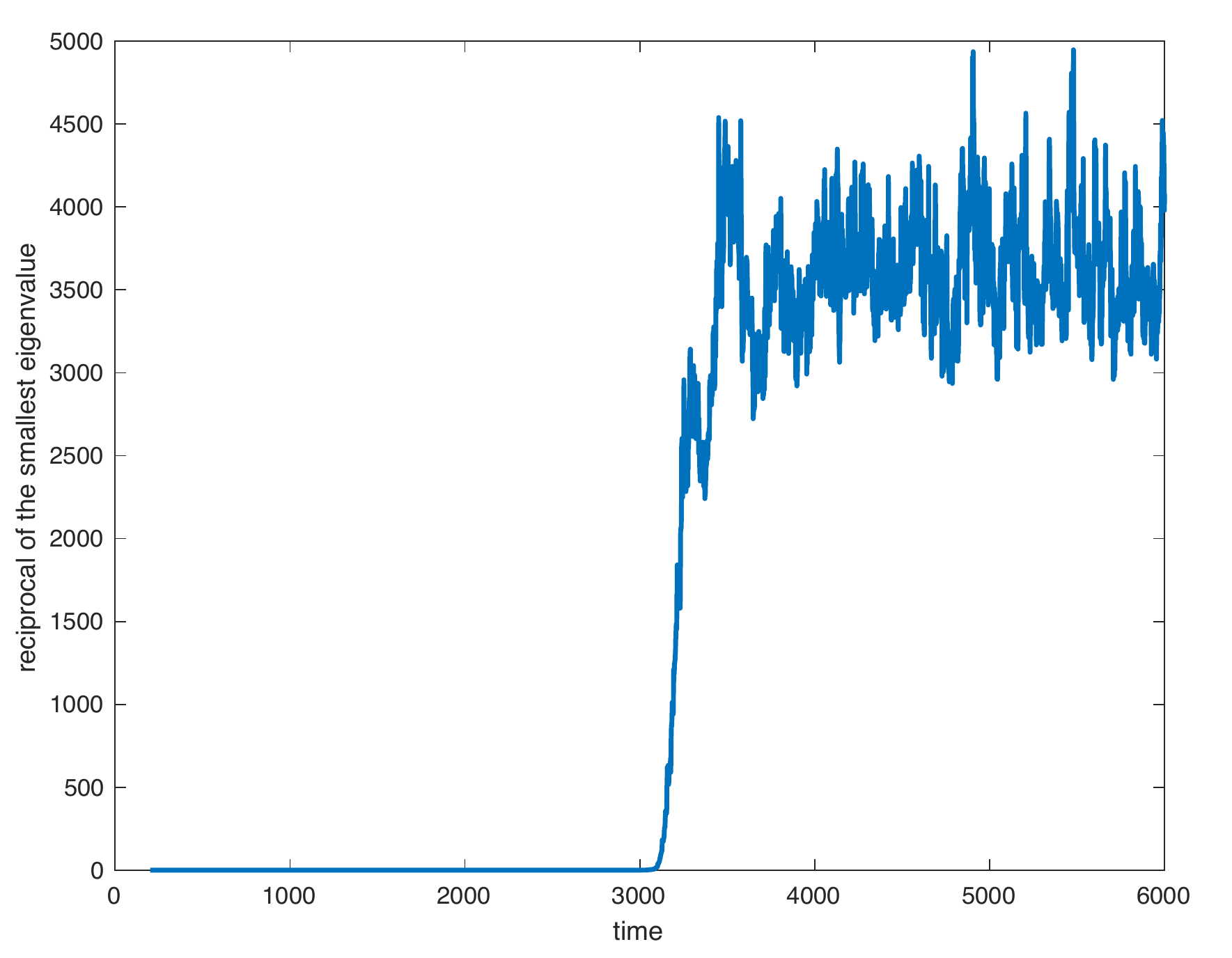}\\ 
(a)&(b)
\end{tabular}
\end{center}
\caption{Plot of detection statistic of $T_2$ in \eqref{procedure2} over time  for (a) random to flock; (b) random to torus. The true change-point happens in the middle of the sequence at time 3000. $T_2$ can detect the change quickly in both cases.}
\label{fig:swarm_detection}
\vspace{-0.2in}
\end{figure}


\clearpage
\bibliographystyle{ieeetr}
\bibliography{CAMSAP2017}

\clearpage

 \begin{proof}[Proof of Proposition \ref{ARL1}]
For problem (1), if we ignore the temporal correlation of scan-statistics and set $\mathbb{P} \left\{\lambda_{\max}(\widehat{\Sigma}_{t-w,t}) \geq b \right\} = \alpha$, then the ARL is roughly $1/\alpha$.

\begin{equation}
\begin{aligned}
& \mathbb{P} \left\{\lambda_{\max}(\widehat{\Sigma}_{t-w,t}) \geq b \right\} \\
= &\mathbb{P} \left\{\frac{ \lambda_{\max}(\widehat{\Sigma}_{t-w,t}) -\mu_{wp}/w}{\sigma_{wp}/w} \geq \frac{b-\mu_{wp}/w}{\sigma_{wp}/w}  \right\} \\
\approx &\mathbb{P} \left\{ W_1 \geq \frac{b -\mu_{wp}/w}{\sigma_{wp}/w} \right\}. \\
= & \alpha = 1/\text{ARL}\\
\end{aligned}
\end{equation}
where $W_1$ is a random variable following Tracy-Widom distribution, and $\mu_{wp}$ and $\sigma_{wp}$ are given in (\ref{eq:parameter}).

Therefore, we have that 
\begin{equation}
 \frac{b -\mu_{wp}/w}{\sigma_{wp}/w} =  \mathcal{T}_{\alpha} \Rightarrow b = \frac{\sigma_{wp}}{w} \mathcal{T}_{\alpha}+\frac{\mu_{wp}}{w} 
 \label{eq:ARLapprox}
\end{equation}
where $\mathcal{T}_\alpha$ is the upper $\alpha$ quantile of the Tracy-Widom distribution, i.e., $\mathbb{P} \left\{ W_1 \geq \mathcal{T}_\alpha \right\} = \alpha$.

Similarly, for problem (2), we have

\begin{equation}
\begin{aligned}
& \mathbb{P} \left\{[\lambda_{\min}(\widehat{\Sigma}_{t-w,t})]^{-1} \geq b \right\} \\
= & \mathbb{P} \left\{\lambda_{\min}(\widehat{\Sigma}_{t-w,t}) \leq 1/b \right\} \\
= &\mathbb{P} \left\{\frac{ \lambda_{\min}(\widehat{\Sigma}_{t-w,t}) -\mu^\prime_{wp}/n}{\sigma^\prime_{wp}/w} \leq \frac{1/b-\mu^\prime_{wp}/w}{\sigma^\prime_{wp}/w}  \right\} \\
\approx &\mathbb{P} \left\{ W_1 \leq \frac{1/b -\mu^\prime_{wp}/w}{\sigma^\prime_{wp}/w} \right\} \\
= &\alpha  =  \frac1{\text{ARL}}. \\
\end{aligned}
\end{equation}
where $W_1$ is a random variable following Tracy-Widom distribution.

Therefore, we have that 
\begin{equation}
 \frac{1/b -\mu^\prime_{wp}/w}{\sigma^\prime_{wp}/w} =  \mathcal{T}_{1-\alpha} \Rightarrow b = \frac1{\frac{\sigma^\prime_{wp}}{w} \mathcal{T}_{1-\alpha}+\frac{\mu^\prime_{wp}}{w} }
\end{equation}
here $\mathcal T_{1-\alpha}$ is the upper $1-\alpha$ quantile of the Tracy-Widom distribution, also known as the $\alpha$ lower quantile of the Tracy-Widom distribution.
\end{proof}

\begin{proof}[Proof of Lemma \ref{CovZ}]
Now we also want to give a general upper bound for the covariance of two eigenvalues $Z_{t}$ and $Z_{t+\delta}$, where
\[
\begin{aligned}
Z_{t}&=\lambda_{\max}(\widehat{\Sigma}_{t-w,t})=\max\limits_{\left\Vert u\right\Vert=1} u^\mathrm{T}(\frac1w\sum\limits_{i=t-w+1}^t x_ix_i^\mathrm{T})u,\\
Z_{t+\delta}&=\lambda_{\max}(\widehat{\Sigma}_{t-w+\delta,t+\delta})=\max\limits_{\left\Vert v\right\Vert=1} v^\mathrm{T}(\frac1w\sum\limits_{i=t-w+\delta+1}^{t+\delta} x_ix_i^\mathrm{T})v
\end{aligned}
\]

Now let 
\[
P=\sum\limits_{i=t-w+1}^{t-w+\delta}x_ix_i^\mathrm{T},  
Q=\sum\limits_{i=t-w+\delta+1}^t x_ix_i^\mathrm{T},  
W=\sum\limits_{i=t+1}^{t+\delta} x_ix_i^\mathrm{T},
\]
then $P$, $Q$ and $W$ are independent random matrices.  
\[
\begin{aligned}
&\mathbb{E}[Z_{t}Z_{t+\delta}]\\
=&\frac1{w^2}\mathbb{E}\left[\! \max\limits_{\left\Vert u\right\Vert=\left\Vert v\right\Vert=1} \! u^\mathrm{T}( \! \sum\limits_{i=t-w+1}^t \! x_ix_i^\mathrm{T})uv^\mathrm{T}(\! \sum\limits_{i=t-w+\delta+1}^{t+\delta} \! x_ix_i^\mathrm{T})v\right]\\
=&\frac1{w^2}\mathbb{E}\left[\max\limits_{\left\Vert u\right\Vert= \left\Vert v\right\Vert=1}u^\mathrm{T}(P+Q)uv^\mathrm{T}(Q+W)v\right]\\
\le&\frac1{w^2}\mathbb{E}\left\{\max\limits_{\left\Vert u\right\Vert = \left\Vert v\right\Vert=1}[u^\mathrm{T}Puv^\mathrm{T}Qv]+\max\limits_{\left\Vert u\right\Vert = \left\Vert v\right\Vert=1}[u^\mathrm{T}Puv^\mathrm{T}Wv]\right.\\
&\quad +\max\limits_{\left\Vert u\right\Vert = \left\Vert v\right\Vert=1}[u^\mathrm{T}Quv^\mathrm{T}Wv]+\max\limits_{\left\Vert u\right\Vert = \left\Vert v\right\Vert=1}[u^\mathrm{T}Quv^\mathrm{T}Qv]\bigg\}\\
=&\frac1{w^2}\left\{\mathbb{E}[\lambda_{\max}(Q)^2] + \mathbb{E}[\lambda_{\max}(P)]E[\lambda_{\max}(W)]\right.\\
&\quad\quad + \mathbb{E}[\lambda_{\max}(Q)]\left(\mathbb{E}[\lambda_{\max}(P)]+\mathbb{E}[\lambda_{\max}(W)]\right)\}\\
\end{aligned}
\]
These mean and variance can be computed using Tracy-Widom distribution as shown in (\ref{eq:parameter}).
Denote $c_1=\mathbb{E}(W_1)=-1.21, c_2 = \sqrt{Var(W_1)}=1.27$. Let $\vartheta = \delta/w$.
Because $p$ is a fixed constant here, we just write $\mu_{n}$ and $\sigma_n$ instead of $\mu_{np}$ and $\sigma_{np}$ to simplify our notation.
\[
\begin{aligned}
\mathbb{E}[Z_{t}Z_{t+\delta}] & \leq \underbrace{\left(\frac{\mu_{w(1-\vartheta)}+c_1\sigma_{w(1-\vartheta)}}{w}\right)^2}_{\uppercase\expandafter{\romannumeral1}}+\underbrace{\left(\frac{c_2\sigma_{w(1-\vartheta)}}{w}\right)^2}_{\uppercase\expandafter{\romannumeral2}}\\
&+\underbrace{2\left[\frac{\mu_{w(1-\vartheta)}+c_1\sigma_{w(1-\vartheta)}}{w}\right]\left(\frac{\mu_{w\vartheta}+c_1\sigma_{w\vartheta}}{w}\right)}_{\uppercase\expandafter{\romannumeral3}}\\
&+\underbrace{\left(\frac{\mu_{w\vartheta}+c_1\sigma_{w\vartheta}}{w}\right)^2}_{\uppercase\expandafter{\romannumeral4}}\\
\end{aligned}
\]
Consider the case when $w/p \rightarrow \infty$, therefore we have:
\begin{align*}
&\frac{\mu_{w(1-\vartheta)}}{w} \\
&= \frac{\left(\sqrt{w(1-\vartheta)-1}+\sqrt{p}\right)^2}{w}\\
&\approx \frac{(\sqrt{w(1-\vartheta)-1})^2(1+\sqrt{p}/\sqrt{w(1-\vartheta)-1})^2}{w}\\
&\approx \frac{w(1-\vartheta)-1}{w}\\
&\approx 1-\vartheta\\
\end{align*}
\begin{align*}
&\frac{\sigma_{w(1-\vartheta)}}{w} \\
&= \frac{\left(\sqrt{w(1-\vartheta)-1}+\sqrt{p}\right)\left(\frac1{\sqrt{w(1-\vartheta)-1}}+\frac1{\sqrt{p}}\right)^\frac13}{w}\\
&\approx \frac{\sqrt{w(1-\vartheta)-1}p^{-\frac16}(1+\frac{\sqrt{p}}{\sqrt{w(1-\vartheta)-1}})^{\frac43}}{w}\\
&\approx \sqrt{\frac{1-\vartheta}{w}}p^{-\frac16}\\
\end{align*}
Now we do Taylor expansion for components \uppercase\expandafter{\romannumeral1} and \uppercase\expandafter{\romannumeral2}. 
For part \uppercase\expandafter{\romannumeral1}, we have:
\[
\begin{aligned}
&\left(\frac{\mu_{w(1-\vartheta)}+c_1\sigma_{w(1-\vartheta)}}{w}\right)^2\\
& \approx  \left(1-\vartheta + c_1\sqrt{\frac{1-\vartheta}{w}}p^{-\frac16}\right)^2 \\
&=  (1-\vartheta)\left(\sqrt{1-\vartheta}+c_1\frac{p^{-\frac16}}{\sqrt{w}}\right)^2\\
&= (1-\vartheta)\left(1-\vartheta+2c_1\frac{p^{-\frac16}}{\sqrt{w}}\sqrt{1-\vartheta}+c_1^2\frac{p^{-\frac13}}{w}\right)\\
&\approx  (1-\vartheta)\left(1-\vartheta+2c_1\frac{p^{-\frac16}}{\sqrt{w}}(1-\frac12\vartheta+o(\vartheta))+c_1^2\frac{p^{-\frac13}}{w}\right)\\
&\approx \left(1+c_1\frac{p^{-\frac16}}{\sqrt{w}}\right)^2 \!\!\!-\!\! \left(1+c_1\frac{p^{-\frac16}}{\sqrt{w}}\right)\!\!\left(2+c_1\frac{p^{-\frac16}}{\sqrt{w}}\right)\vartheta \!+\! o(\vartheta)\\
\end{aligned}
\]
For part \uppercase\expandafter{\romannumeral2},
\[
\left(\frac{c_2\sigma_{w(1-\vartheta)}}{w}\right)^2 \approx c_2^2 \frac{1-\vartheta}{w}p^{-\frac13}
\]
Since our main focus is the local covariance structure, $\vartheta$ is small. Therefore parts \uppercase\expandafter{\romannumeral3} and \uppercase\expandafter{\romannumeral4} will vanish. 
In total, we have
\begin{align*}
{\rm corr}(Z_t,Z_{t+\delta}) 
&=\frac{\mathbb{E}[Z_t Z_{t+\delta}] - \mathbb E(Z_t) \mathbb E(Z_{t+\delta})}{\sqrt{{\rm Var}(Z_t)}\sqrt{{\rm Var}(Z_{t+\delta})}}\\
\lesssim & \frac1{\left(\frac{c_2p^{-\frac16}}{\sqrt{w}}\right)^2}\Bigg\{\left(1+c_1\frac{p^{-\frac16}}{\sqrt{w}}\right)^2  + c_2^2 \frac{1-\vartheta}{w}p^{-\frac13}\\
- & \left(1+c_1\frac{p^{-\frac16}}{\sqrt{w}}\right)\left(2+c_1\frac{p^{-\frac16}}{\sqrt{w}}\right)\vartheta\\
-&\left(1+ \frac{c_1}{\sqrt{w}}p^{-\frac16}\right)^2 + o(\vartheta)\Bigg\}
\end{align*}
\begin{align*}
=& 1-\left(1+\frac{\left(1+c_1\frac{p^{-\frac16}}{\sqrt{w}}\right)\left(2+c_1\frac{p^{-\frac16}}{\sqrt{w}}\right)}{\frac{c_2^2}{w}p^{-\frac13}}\right)\vartheta + o(\vartheta)\\
=& 1 - \left(1+\frac{2p^{\frac13}+3p^{\frac16}\frac{c_1}{\sqrt{w}}+\frac{c_1^2}{w}}{c_2^2}\right)\delta + o(\delta)\\
\end{align*}

\end{proof}

\begin{proof}[Proof of Proposition \ref{EDD}]
To prove this theorem, we first relate the detection procedure to a CUSUM procedure, note that 
\begin{equation}
\lambda_{\max}(\widehat{\Sigma}_{k,t}) = \max_{q\in \mathbb{R}^{p\times 1}, \left\|q\right\|=1}
|q^\intercal \widehat{\Sigma}_{k, t} q|
\label{eq:q}
\end{equation}
For each $q$, we have
\[
(t-k)|q^\intercal \widehat{\Sigma}_{k, t} q| =  \sum_{i=k+1}^t (q^\intercal x_i)^2
\]
According to the Grothendieck's Inequality (see \cite{guedon2016community}), the $q$ that attain the maximum in equation (\ref{eq:q}) is very close to $u$ under $\textsf{H}_1$. Therefore, assuming the optimal $q$ always equals to $u$ will only cause a small error but will bring great convenience to our analysis. 

Now we have under $\textsf{H}_0$, $q^\intercal x_i \sim \mathcal{N}(0, 1)$ and under $\textsf{H}_1$, $q^\intercal x_i \sim \mathcal{N}(0, 1+\theta)$. Let $f_0$ denote the pdf of $\mathcal{N}(0, 1)$ and $f_1$ the pdf of $\mathcal{N}(0, 1+\theta)$. we first compute the log-likelihood ratio:
\[
\log\frac{f_1(y)}{f_0(y)} = -\frac12\log(1+\theta)+\frac12(1-\frac1{1+\theta})y^2.
\]
Therefore, the CUSUM procedure looks like
\[
\widetilde{T} \!=\! \inf\{t \!: \! \max_{k<t} \!\! \sum_{i=k+1}^t \! \left[\frac12(1\!-\! \frac1{1+\theta})(q^\intercal x_i)^2 \!-\! \frac{\log(1\!+\!\theta)}{2}\right] \! \geq \! b'\}
\]
Compare with the detection procedure in (\ref{procedure1}), by letting $b' = \frac12(1-\frac1{1+\theta})(b-\frac{\log(1+\theta)}{1-1/(1+\theta)})w$, we can have
\[
\mathbb{E}^1[T] \geq \mathbb{E}^1[\widetilde{T}].
\]
Since $\widetilde{T}$ is a CUSUM procedure, whose properties are well understood, we may obtain its EDD as follows(see details in \cite{siegmund2013sequential}):
\[
\mu_1=\int\log(\frac{f_1(y)}{f_0(y)})f_1(y)dy=-\frac12\log(1+\theta)+\frac{\theta}{2},
\]
\[
\mathbb{E}^1[\widetilde{T}]=\frac{e^{-b^\prime}+b^\prime-1}{-\frac12\log(1+\theta)+\frac{\theta}{2}}.
\]
\end{proof}


\noindent{\it Remark:} Another common strategy to compute the detection delay is using Kullback-Leibler(KL) divergence. KL divergence between two continuous distributions with probability density function $p(x)$ and $q(x)$ is defined as
\[
D_{KL}\left( p(x) \| q(x) \right) \triangleq \int p(x) \log \frac{p(x)}{q(x)} dx,
\]
Therefore we can compute the KL divergence between $\mathcal{N}(0, I_p)$ and $\mathcal{N}(0,  I_p + \theta U U^\top)$:
\[
\begin{aligned}
&D_{KL}\left( \mathcal{N}(0, I_p) \| \mathcal{N}(0,  I_p + \theta U U^\top) \right) \\
&=\frac12\left[\log\frac{\|I_p + \theta U U^\top\|}{\|Ip\|}-p+\mbox{tr}((I_p + \theta U U^\top)^{-1}I_p)\right]\\
&=\frac12\left\{\log(1+\theta)-p+\mbox{tr}\left[(I_p + \theta U U^\top)^{-1}\right]\right\}\\
&=\frac12\left[\log(1+\theta)-p+\mbox{tr}\left(I_p - \frac{U U^\top\theta}{1+\theta}\right)\right]\\
&=\frac12\left[\log(1+\theta)-p+p-\frac{\theta}{1+\theta}\right]\\
&= \frac12 \left[ \log(1+\theta) - \frac{\theta}{1+\theta} \right]
\end{aligned}
\]
here in the third equality using Sherman-Morrison formula \cite{akgun2001fast}. 
The detection delay is greater than \[\frac{2\log(\text{ARL})}{\log(1+\theta)-\theta/(1+\theta)}.\]


\end{document}